\newif\ifdraft
\newif\ifarxiv
\arxivtrue 

\ifarxiv
\documentclass[12pt]{amsart}
\usepackage[margin=1.3in,marginparwidth=1in]{geometry}
\else
\documentclass{proc-l}
\fi

\usepackage{enumitem}
\usepackage{verbatim}
\usepackage{xcolor}
\usepackage{graphicx}
\usepackage[normalem]{ulem}
\usepackage{tikz}
\usetikzlibrary{intersections,positioning,calc}
\usepackage{hyperref}

\ifdraft
\usepackage{fancyhdr,datetime2}
\fancypagestyle{firstpage}{%
	
	\fancyhf{}
	\fancyhead[C]{\color{red}\fbox{\textbf{DRAFT} (\DTMnow) -- NOT FOR DISTRIBUTION}}
	\fancyfoot[C]{\thepage}
}
\fi

\numberwithin{equation}{section}

\usepackage{mymacros}

\begin{document}

\title{Gibbs measures have local product structure}
\author{Vaughn Climenhaga}
\address{Dept.\ of Mathematics, University of Houston, Houston, TX 77204}
\email{climenha@math.uh.edu}
\date{\today}
\thanks{This material is based upon work supported by the National Science Foundation under Award No.\ DMS-1554794 and DMS-2154378.}
\begin{abstract}
It is well-known that equilibrium measures for uniformly hyperbolic dynamical systems have a local product structure, which plays an important role in their mixing properties. Existing proofs of this fact rely either on transfer operators or on leafwise constructions, and in particular are not well-suited to the approach to thermodynamic formalism based on Bowen's specification property. Here we provide an alternate proof based on the Gibbs property, which fits more comfortably in that approach.
\end{abstract}
\maketitle
\ifdraft
\thispagestyle{firstpage}
\pagestyle{fancy}
\fi

\section{Introduction and main results}

Let $M$ be a compact smooth Riemannian manifold and $f\colon M\to M$ a $C^1$ diffeomorphism that is Anosov, meaning that the tangent bundle splits as $TM = E^u \oplus E^s$, where $E^s$ and $E^u$ are invariant subbundles such that for some $\ell\in\NN$, $\|Df^\ell|_{E^s}\| < 1$ and $\|Df^{-\ell}|_{E^u}\| < 1$.

The theory of thermodynamic formalism developed in the 1970s considers a \emph{potential function} $\ph\colon M\to \RR$ and studies \emph{equilibrium measures} maximizing the quantity $h_\mu(f) + \int \ph\,d\mu$ over the space of all $f$-invariant Borel probability measures on $M$. (Here $h_\mu(f)$ is Kolmogorov--Sinai measure-theoretic entropy.) If $\ph$ is H\"older continuous and $f$ is topologically transitive, meaning that it has a dense orbit, then there is a unique equilibrium measure $\mu = \mu_\ph$; the class of such measures includes important special cases such as the measure of maximal entropy (MME) and (when $f$ is $C^{1+\alpha}$) the Sinai--Ruelle--Bowen (SRB) measure.

In this uniformly hyperbolic setting, an important property of equilibrium measures is their \emph{local product structure}, which reflects the product structure of the system itself: given $x\in M$, the local stable and unstable leaves through $x$ introduce a coordinate system on a neighborhood $R$ of $x$, in which $\mu|_R$ is absolutely continuous with respect to a product measure (see Definition \ref{def:lps} below). This plays an important role in the strong stochastic properties satisfied by equilibrium measures, such as ergodicity, mixing, K, Bernoulli, exponential decay of correlations, central limit theorem, etc.; see for example \cite{CHT16, nA22, CP23}. 

Local product structure for the MME was shown by Margulis \cite{gM70}, and for more general equilibrium measures the classical construction via eigenfunctions and eigenmeasures of the Ruelle--Perron--Frobenius operator \cite{Bow08} can be understood as a kind of product structure. A full description of the product structure was given in \cite{Hay94,Lep00}, again using the RPF operator. For geodesic flows in negative curvature, the Patterson--Sullivan construction of equilibrium measures \cite{vK90} produces a product structure. One can also describe the product structure in terms of a dimension-theoretic leafwise construction: see \cite{uH89,bH89} for the MME case, and \cite{CPZ19,CPZ20,Cli24,CD24} for the general case.

Another important approach to equilibrium measures is due to Rufus Bowen.
For Anosov systems, H\"older continuous potentials satisfy a bounded distortion condition now called the \emph{Bowen property}: there are $r,L>0$ such that
\begin{equation}\label{eqn:intro-Bow}
\text{for all $x\in M$, $n\in \NN$, and $y\in B_n(x,r)$, we have $|S_n\ph(x) - S_n\ph(y)| \leq L$}.
\end{equation}
Here $S_n\ph(x) := \sum_{k=0}^{n-1} \ph(f^k x)$ is the $n$th ergodic sum of $\ph$ at $x$, and
\[
\text{$B_n(x,r) := \{y\in M : d(f^k x, f^k y) < r$ for all $0\leq k < n\}$}
\]
is the radius $r$ \emph{Bowen ball} around $x$ of order $n$.
Transitive Anosov systems have the \emph{specification} and \emph{expansivity} properties \cite{rB71}; together with continuity and the Bowen property, these imply existence of a unique equilibrium measure \cite{Bow75}.
However, the author is not aware of any results in the literature that use these properties to establish local product structure. In particular, more recent generalizations of this approach to the non-uniformly hyperbolic setting \cite{CT16,BCFT18,CT21,PYY} have been silent on the question of local product structure of the equilibrium measures they produce.

The purpose of this paper is to provide a direct proof of local product structure using the Gibbs property of equilibrium measures, which \emph{can} be established directly using the specification, expansivity, and Bowen properties \cite{Bow75}, sidestepping the use of transfer operators or leafwise constructions. One important motivation is the expectation that this argument will prove fruitful in future adaptations to non-uniform hyperbolicity; see \cite{CCESW} for recent progress in this direction.

Our main result (Theorem \ref{thm:main} below) immediately implies the following.

\begin{theorem}\label{thm:Holder}
If $f\colon M\to M$ is a topologically transitive $C^1$ Anosov diffeomorphism and $\ph \colon M\to \RR$ is continuous and satisfies the Bowen property \eqref{eqn:intro-Bow}, then the unique equilibrium measure has local product structure with uniformly bounded densities in the sense of Definition \ref{def:lps}.
\end{theorem}

It should be emphasized that the main innovation here is not the result itself, which can also be deduced from \cite{CPZ20} (see Remark \ref{rmk:lit}), but rather the simpler method of proof using the following \emph{Gibbs property}, which Bowen proved in \cite{Bow75} for the unique equilibrium measure in Theorem \ref{thm:Holder}.

\begin{definition}\label{def:Gibbs}
We say that a probability measure $\mu$ on $M$ has the \emph{Gibbs property} if there are $P\in \RR$ and $r_0>0$ such that for every $r\in (0,r_0]$, there is a constant $K=K(r)>0$ such that
\begin{equation}\label{eqn:gibbs}
K^{-1} \leq \frac{\mu(B_n(x,r))}{e^{S_n\ph(x) - nP}} \leq K
\quad\text{ for all } x,n.
\end{equation}
\end{definition}
 
The rest of this paper is devoted to the proof of the following, which (together with \cite{Bow75}) immediately implies Theorem \ref{thm:Holder}.

\begin{theorem}\label{thm:main}
Let $f\colon M\to M$ be a $C^1$ Anosov diffeomorphism, $\ph\colon M\to \RR$ a continuous function, and $\mu$ an $f$-invariant Borel probability measure on $M$ satisfying the Gibbs property \eqref{eqn:gibbs} for some $P\in \RR$.  Then $\mu$ has local product structure with uniformly bounded densities in the sense of Definition \ref{def:lps}.
\end{theorem}

Theorem \ref{thm:main} does not explicitly state the Bowen property as an assumption, but it should be pointed out that the Gibbs property assumed in \eqref{eqn:gibbs} can only hold if $\ph$ satisfies the Bowen property; see Lemma \ref{lem:Bow-prop}. 

\begin{remark}
There should be no obstruction to proving Theorem \ref{thm:main} for Axiom A diffeomorphisms or locally maximal hyperbolic sets, but here we focus on the simplest possible setting in order to minimize notation and complications, and to make the underlying ideas as transparent as possible. Similarly, we do not pursue the natural and important question of what happens if $\mu$ has a weaker version of the Gibbs property; \cite{CCESW} demonstrates that in some such settings one may still expect to get product structure by adapting the ideas developed here.
\end{remark}

\begin{remark}\label{rmk:lit}
Theorem \ref{thm:Holder} can also be deduced from \cite{CPZ20}, which covers a broader class of partially hyperbolic systems; the regularity of the diffeomorphism is not explicitly stated, but $C^1$ suffices. The local product structure considered in \cite[Definition 2.13]{CPZ20} omits the ``uniformly bounded densities'' property in Definition \ref{def:lps} here, but that property can be deduced from the proofs there.

Similar results were also proved in \cite{Lep00,CPZ19}, both of which assume that $f$ is $C^{1+\alpha}$ and $\ph$ is H\"older, although the primary role of the H\"older continuity assumption on $Df$ appears to be the application of the results to the SRB measure (which we are not concerned with here), rather than the product structure property for other equilibrium meaures. H\"older continuity of $\ph$ is used in a more essential way in \cite{Lep00} to apply RPF operator methods, but not in \cite{CPZ19}, which only uses it to deduce the Bowen property.
\end{remark}

\subsection*{Acknowledgments}
The author is grateful to Caleb Dilsavor for pointing out an error in an earlier version of the introduction, and to the anonymous referee for a careful reading and pointing out some issues in the initial arguments.

\section{Local product structure}\label{sec:lps}

In this section we recall basic definitions and results from hyperbolic dynamics; then in \S\S\ref{sec:Bow-ball-prod}--\ref{sec:finish} we prove Theorem \ref{thm:main}.

\subsection{Stable and unstable manifolds}\label{sec:mfds}

Before defining local product structure for $\mu$, we recall some standard definitions and facts; see \cite[Chapter 6]{KH95} for details and proofs.  Since $f$ is Anosov, every $x\in M$ lies on $C^1$ local stable and unstable manifolds $W^{s,u}_x \subset M$ tangent to $E^{s,u}_x$, which are local leaves of continuous $f$-invariant foliations of $M$, and which have the following properties.
\begin{enumerate}
\item There are $C \geq 1$ and  $\lambda \in (0,1)$ such that for all $x\in M$, $y\in W^s_x$, and $n\geq 0$, we have $d(f^n x, f^n y) \leq C\lambda^n d(x,y)$; a similar contraction bound holds going backwards in time when $y\in W^u_x$.
\item There is $\delta>0$ such that if $d(f^n x, f^ny) \leq \delta$ for all $n\geq 0$, then $y\in W^s_x$; similarly for $W^u_x$ with $n\leq 0$.
\item There is $\eps>0$ such that if $d(x,y) \leq \eps$, then $W_x^s \cap W_y^u$ is a single point, which we denote $[x,y]$.  Moreover, there is a constant $Q$ such that $d([x,y],x) \leq Q d(x,y)$ for all $x,y$, and similarly for $d([x,y],y)$.
\end{enumerate}
When $x,y$ are sufficiently close, the bracket can be characterized as follows.

\begin{lemma}\label{lem:bracket-dyn}
If $x,y\in M$ satisfy $d(x,y) < \delta/(CQ)$, then there is exactly one point $z\in M$ such that
\begin{equation}\label{eqn:zxy}
\text{for all $k\geq 0$, we have } d(f^k z, f^k x) < \delta
\text{ and } d(f^{-k} z, f^{-k} y) < \delta,
\end{equation}
and that point is $z=[x,y]$.
\end{lemma}
\begin{proof}
If the point $z$ satisfies \eqref{eqn:zxy} then $z\in W_x^s \cap W_y^u$, so $z=[x,y]$. Conversely, if $z=[x,y]$, then $d(z,x) \leq Qd(x,y) < \delta/C$ and $d(z,y) \leq Qd(x,y) < \delta/C$, which implies \eqref{eqn:zxy}.
\end{proof}

We will need the following consequence of this.

\begin{lemma}\label{lem:bracket-inv}
Fix $\eps_1 \in (0,\eps]$ such that $\max(C, \|Df\|) Q\eps_1 < \delta$.
If $x,y \in M$ and $n\in \NN$ are such that
$d(f^k x, f^k y) \leq \eps_1$ for all $0\leq k\leq n$, then $f^n([x,y]) = [f^n x, f^n y]$.
\end{lemma}
\begin{proof}
It suffices to prove this in the case $n=1$; the general case then follows by induction. If $d(x,y) \leq \eps_1$, then by Lemma \ref{lem:bracket-dyn}, $z=f[x,y]$ satisfies
\[
\text{for all $k\geq 0$, $d(f^k z, f^k(fx)) = d(f^{k+1}[x,y], f^{k+1}x) < \delta$},
\]
and similarly,
\[
\text{for all $k \geq 1$, $d(f^{-k} z, f^{-k}(fy)) = d(f^{-(k-1)}[x,y], f^{-(k-1)} y) < \delta$}.
\]
Finally, $d(z, fy) = d(f[x,y], fy) \leq \|Df\| d([x,y], y) \leq \|Df\| Q \eps_1 < \delta$, so by Lemma \ref{lem:bracket-dyn}, we have $f[x,y] = [fx, fy]$.
\end{proof}

Let $\Delta_\eps := \{(x,y) \in M\times M : d(x,y) \leq \eps\}$ be the set of pairs where the bracket is guaranteed to be defined.
By continuity of the stable and unstable foliations, the bracket map $\Delta_\eps \to M$ is continuous, and since $\Delta_\eps$ is compact, it is uniformly continuous. We record this in the following form.

\begin{lemma}\label{lem:bracket-cts}
There is a function $\omega \colon (0,\infty) \to (0,\infty)$ such that 
$\lim_{r\to 0^+} \omega(r) = 0$, 
with the property that
if $(x,y), (x',y') \in \Delta_\eps$ 
are such that $d(x,x') <r$ and $d(y,y')<r$, 
then $d([x,y], [x',y']) < \omega(r)$. (See Figure \ref{fig:brackets}(a).)
\end{lemma}

For future reference, we observe that the definition immediately gives
\begin{equation}\label{eqn:bracket-twice}
[[x,y], [x',y']] = [x,y']
\end{equation}
whenever all the brackets involved are defined; see Figure \ref{fig:brackets}(a). 

\begin{figure}[tbp]
\quad
\begin{tikzpicture}[baseline, scale=.6, bend angle=8, inner sep=1pt, label distance=1pt,
					arr/.style={<-,draw=gray,shorten <=2pt, >=stealth},
					dot/.style={fill=blue,circle,minimum size=3pt}]
\coordinate (y'W) at (-2,1);
\coordinate (y'C) at (0,1);
\coordinate (y'E) at (2,1);
\path [name path=Wy'u, draw]
	(y'W)  node[black!70,below left]{$W_{y'}^u$} to [bend left] (y'C) to [bend right] (y'E);
\path [name path=vert'] (-1.2,0)-- ++(0,2);
\path [name intersections={of=Wy'u and vert', by=y'}];
\node[dot] at (y') [label=-60:$y'$] {};

\coordinate (yW) at (-2,1.3);
\coordinate (yC) at (0,1.6);
\coordinate (yE) at (2,1.9);
\path [name path=Wyu, draw]
	(yW) node[black!70,above left]{$W_{y}^u$} to [bend left] (yC) to [bend right] (yE);
\path [name path=vert] (-1.5,0)-- ++(0,2);
\path [name intersections={of=Wyu and vert, by=y}];
\node[dot] at (y) [label=above:$y$] {};

\coordinate (xS) at (0,-2.4);
\coordinate (xC) at (.8,0);
\coordinate (xN) at (1.6,2.4);
\path [name path=Wxs, draw]
	(xS)  node[black!70,below right]{$W_{x}^s$} to [bend left] (xC) to [bend right] (xN);
\path [name path=hor] (-1,-1.4)-- ++(2,0);
\path [name intersections={of=Wxs and hor, by=x}];
\node[dot] at (x) [label=right:$x$] {};

\coordinate (x'S) at (-.3,-2.4);
\coordinate (x'C) at (0,0);
\coordinate (x'N) at (0.5,2.4);
\path [name path=Wx's, draw]
	(x'S) node[black!70,below left]{$W_{x'}^s$} to [bend left] (x'C) to [bend right] (x'N);
\path [name path=hor'] (-1,-1.7)-- ++(2,0);
\path [name intersections={of=Wx's and hor', by=x'}];
\node[dot] at (x') [label=left:$x'$] {};

\path [name intersections={of=Wxs and Wyu, by=xy}];
\path [name intersections={of=Wx's and Wy'u, by=x'y'}];
\path [name intersections={of=Wxs and Wy'u, by=xy'}];
\draw[arr] (xy) node[dot]{} to [bend left] (2,2.6) node[above right]{$[x,y]$};
\draw[arr] (x'y') node[dot]{} to [bend right] (-.2,2.6) node[above left]{$[x',y']$};
\draw[arr] (xy') node[dot]{} to (2.3, 1.5) node[right]{$[x,y']$};

\draw[red] (y)--(y');
\draw[red] (x)--(x');
\draw[green!20!black] (xy)--(x'y');

\node[fill=red!30!white,rectangle,inner sep=2pt] (r) at (-2,-1) {If $d< r$};
\node[fill=green!50!white,rectangle,inner sep=2pt] (omega) at (2.6,-.5) {then $d<\omega$};

\draw[arr] ($.5*(y)+.5*(y')$) to [bend right=20](r);
\draw[arr] ($.5*(x)+.5*(x')$) to [bend right=20](r);
\draw[arr] ($.5*(xy)+.5*(x'y')$) to [bend right=20](omega);
\end{tikzpicture}
\hfill
\begin{tikzpicture}[baseline, scale=.6, bend angle=8, inner sep=1pt, 
					node distance=4pt, label distance=0pt, 
					arr/.style={<-,shorten <=2pt, >=stealth, semithick},
					dot/.style={fill=blue,circle,minimum size=3pt}]
\coordinate (N) at (1,2);
\coordinate (E) at (2,0);
\coordinate (q) at (0,0);
\coordinate (S) at ($-1.3*(N)$);
\coordinate (W) at ($-1*(E)$);
\coordinate (NE) at ($(N) + (E)$);
\coordinate (SE) at ($(S) + (E)$);
\coordinate (SW) at ($(S) + (W)$);
\coordinate (NW) at ($(N) + (W)$);

\path [name path=R, fill=yellow!30!white, draw]
	(NW) to [bend left] (N) to [bend right] (NE)
	to [bend left] (E) to [bend right] (SE)
	to [bend left] (S) to [bend right] (SW)
	to [bend left] (W) to [bend right] cycle;
\path [name path=Vqu, draw, line width=1pt] 
	(W) to [bend left] (q) to [bend right] (E);
\path [name path=Vqs, draw, line width=1pt]
	(S) to [bend left] (q) to [bend right] (N);
\node[left=of W] {$V_q^u$};
\node[above=of N] {$V_q^s$};
\node[dot] at (q) [label=-80:$q$] {};

\coordinate (xo) at ($.7*(E)$);
\path [name path=ExN] (xo) -- ++($2*(N)$);
\path [name intersections={of=R and ExN, by=xN}];
\path [name path=ExS] (xo) -- ++($2*(S)$);
\path [name intersections={of=R and ExS, by=xS}];
\path [name path=Ex] ($(xo) + (S)$) -- ++($2*(N)$);
\path [name intersections={of=Vqu and Ex, by=x}];
\path [name path=Vxs, draw]
	(xS) to [bend left] (x) to [bend right] (xN);
\node[above=of xN] {$V_x^s$};
\node[dot] at (x) [label=120:$x$] {};

\coordinate (yo) at ($.7*(N)$);
\path [name path=EyW] (yo) -- ++($2*(W)$);
\path [name intersections={of=R and EyW, by=yW}];
\path [name path=EyE] (yo) -- ++($2*(E)$);
\path [name intersections={of=R and EyE, by=yE}];
\path [name path=Ey] ($(yo) + (W)$) -- ++($2*(E)$);
\path [name intersections={of=Vqs and Ey, by=y}];
\path [name path=Vyu, draw]
	(yW) to [bend left] (y) to [bend right] (yE);
\node[left=of yW] {$V_y^u$};
\node[dot] at (y) [label=-160:{$y$}] {};

\path [name intersections={of=Vxs and Vyu, by=z}];
\draw[arr] (z) node[dot]{} to [bend right] (3.1,1) node[right]{$z = [x,y]$};
\node at (3.1,-0.3)[right] {$x = [z,q] = \pi^u(z)$};
\node at (3.1,-1.3)[right] {$y = [q,z] = \pi^s(z)$};
\end{tikzpicture}
\quad\\[-1cm]
(a)\hspace{5cm}(b)\hspace{2cm}
\caption{Continuity of brackets and structure of rectangles}
\label{fig:brackets}
\end{figure}

\subsection{Rectangles}\label{sec:rect}

A set $R\subset M$ is a \emph{rectangle} if it has diameter $\leq\eps$ and is closed under the bracket operation: in other words, for every $x,y\in R$, the intersection point $[x,y] = W_x^s \cap W_y^u$ exists and is contained in $R$. Figure \ref{fig:brackets}(b) illustrates a rectangle together with the standard procedure for producing rectangles: if we are given $q\in M$, $A\subset W_q^u$, and $B\subset W_q^s$ such that $[x,y]$ is defined for every $x\in A$ and $y\in B$, then we write
\begin{align*}
[A,B] &= \{ [x,y] : x\in A, y \in B \}
= \Big(\bigcup_{x\in A} W_x^s \Big) \cap \Big( \bigcup_{y\in B} W_y^u \Big) \\
&= \{ z \in M : [x,z] \in A, [z,x] \in B \}.
\end{align*}
If $A$ and $B$ are in a small enough neighborhood of $q$, this always produces a rectangle:

\begin{lemma}\label{lem:rect}
If $A \subset W_q^u \cap B(q,\eps/2(Q+1))$ and $B\subset W_q^s \cap B(q,\eps/2(Q+1))$, then $[x,y]$ is defined for every $x\in A$ and $y\in B$, and $[A,B]$ is a rectangle.
\end{lemma}
\begin{proof}
Given any $x\in A$ and $y\in B$ we have $d(x,y) < \eps/(Q+1)$, so $[x,y]$ exists, and 
\[
d([x,y],q) \leq d([x,y],x) + d(x,q)  \leq (Q+1)d(x,q) \leq \tfrac\eps2,
\]
so $[A,B]$ has diameter $\leq \eps$. Moreover, given $[x,y], [x',y']\in [A,B]$, where $x,x' \in A$ and $y,y'\in B$, we see from \eqref{eqn:bracket-twice} that $[[x,y],[x',y']] = [x,y'] \in [A,B]$.
\end{proof}

Given a rectangle $R$ and a point $q\in R$, we will write $V_q^u = W_q^u \cap R$ and $V_q^s = W_q^s \cap R$. Observe that $R = [V_q^u, V_q^s]$, so that in fact \emph{every} rectangle can be produced via the bracket operation; they are exactly the sets that exhibit ``product structure'' with respect to the stable and unstable leaves.

\subsection{Measures with local product structure}\label{sec:meas-lps}

Given a rectangle $R$, a point $q\in R$, and Borel measures $\nu^{s,u}$ on $V^{s,u}_q$, let $\nu^u\otimes\nu^s$ be the pushforward of $\nu^s\times \nu^u$ under the bracket map; that is, for every pair of Borel sets $A\subset V_q^u$ and $B\subset V_q^s$, we put $(\nu^u\otimes \nu^s)([A,B]) = \nu^u(A)\nu^s(B)$, and then extend to all Borel sets in $R$ via Carath\'eodory's extension theorem.

\begin{definition}\label{def:lps}
A measure $\mu$ has \emph{local product structure} with respect to $W^{s,u}$ if there is $\eps_1>0$ such that for every $q\in M$ and rectangle $R\ni q$ with $\diam R < \eps_1$, there are measures $\nu^{s,u}$ on $V^{s,u}_q$ such that $\mu|_R \ll \nu^u \otimes \nu^s$.

We say in addition that $\mu$ has \emph{uniformly bounded densities} if there is $\bar K \geq 1$ such that for all $R$, the measures $\nu^{s,u}$ can be chosen so that the Radon--Nikodym derivative $\psi = \frac{d\mu}{d(\nu^u\otimes \nu^s)}$ satisfies $\bar K^{-1} \leq \psi \leq \bar K$ at $\mu$-a.e.\ point.
\end{definition}

\begin{remark}\label{rmk:Holder-psi}
When $f$ is $C^{1+\alpha}$ and $\mu$ is the equilibrium measure for a H\"older continuous potential, one can in fact prove that the Radon--Nikodym derivative $\psi$ from Definition \ref{def:lps} is H\"older, and give an explicit formula. Here we only prove what Call and Park \cite{CP23} refer to as \emph{measurable local product structure}.
\end{remark}

Given a rectangle $R$, let $\pi^{s,u} \colon R \to V_q^{s,u}$ be the projection maps along holonomies, as illustrated in Figure \ref{fig:brackets}(b):
\begin{equation}\label{eqn:pisu}
\pi^s(z) = [q,z] \in V_q^s \cap V_z^u
\quad\text{and}\quad
\pi^u(z) = [z,q] \in V_z^s \cap V_q^u.
\end{equation}
Then consider the following measures on $V_q^{s,u}$:
\begin{equation}\label{eqn:musu}
\mu^s := \pi^s_* \mu|_R,
\qquad
\mu^u := \pi^u_* \mu|_R.
\end{equation}

\begin{lemma}\label{lem:get-leaf-meas}
$\mu$ has local product structure on $R$ if and only if $\mu|_R \ll \mu^u \otimes \mu^s$.
\end{lemma}
\begin{proof}
One direction is immediate from the definition. For the other, we must show that if $\mu|_R \ll \nu^u \otimes \nu^s$ for some measures $\nu^{s,u}$ on $V_q^{s,u}$, then $\mu|_R \ll \mu^u \otimes \mu^s$. Use the Lebesgue decomposition theorem to write $\nu^{s,u} = m^{s,u} + \rho^{s,u}$, where $m^{s,u} \ll \mu^{s,u}$ and $\rho^{s,u} \perp \mu^{s,u}$. 

Let $\psi$ be the Radon--Nikodym derivative of $\mu|_R$ with respect to $\nu^u\otimes\nu^s$.
Since $\rho^u \perp \mu^u$, there exists $E\subset V_q^u$ such that $\mu^u(E) = 0$ (hence $m^u(E) = 0$) and $\rho^u(E^c) = 0$. We have
\[
\mu^u(E) = \mu((\pi^u)^{-1}(E)) = \int_{(\pi^u)^{-1}(E)} \psi \,d(\nu^u \otimes \nu^s)
= \int_E \int_{V_q^s} \psi([y,z]) \,d\nu^s(z) \,d\rho^u(y),
\]
and since $\mu^u(E) = 0$ this implies that $\psi \equiv 0$ $(\rho^u \otimes \nu^s)$-a.e.\ on $(\pi^u)^{-1}(E)$. Since the complement of this set is $(\rho^u \otimes \nu^s)$-null, we conclude that $\psi \equiv 0$ $(\rho^u \otimes \nu^s)$-a.e.\ on $R$. A similar argument applies with the roles of stable and unstable reversed, and we conclude that in the expression
\[
\psi\, d(\nu^u \otimes \nu^s) = \psi\, d(m^u \otimes m^s)
+ \psi\, d(\rho^u \otimes m^s) + \psi \, d(m^u \otimes \rho^s) + \psi \, d(\rho^u \otimes \rho^s),
\]
the last three terms vanish, and thus $\mu \ll m^u \otimes m^s \ll \mu^u \otimes \mu^s$.
\end{proof}

Motivated by Lemma \ref{lem:get-leaf-meas}, we will prove Theorem \ref{thm:main} by proving that every invariant Gibbs measure satisfies $\mu|_R \ll \mu^u \otimes \mu^s$. In fact, we will prove that it has uniformly bounded densities: there exists $\bar{K}\geq 1$ such that for all sufficiently small rectangles $R$ and points $q\in R$, we have
\begin{equation}\label{eqn:unif-bdd}
\bar{K}^{-1} (\mu^u\otimes\mu^s)(E)
\leq \mu(E) \leq \bar{K} (\mu^u\otimes\mu^s)(E)
\text{ for all } E\subset R.
\end{equation}
Our strategy will be as follows.
\begin{enumerate}
\item In \S\ref{sec:Bow-ball-prod}, prove that ``two-sided'' Bowen balls in $R$ can be related to the direct product of one-sided ``leafwise'' Bowen balls $B^u \subset V_q^u$ and $B^s \subset V_q^s$.
\item In \S\ref{sec:quasi-indep}, use this result together with the Gibbs property and invariance of $\mu$ to 
obtain Gibbs bounds on the measures of $B^u$, $B^s$, and $[B^u,B^s]$.
\item In \S\ref{sec:finish}, use these Gibbs bounds to prove the quasi-multiplicativity result \eqref{eqn:unif-bdd} when $E$ is the product of sets in $V_q^{u,s}$, and then extend to arbitrary measurable $E\subset R$.
\end{enumerate}

\section{Leafwise Bowen balls and their products}\label{sec:Bow-ball-prod}

Given $q\in M$ and $r>0$, we will consider the following \emph{leafwise} Bowen balls:
\begin{align*}
B_m^u(q,r) &= \{x\in W_q^u : d(f^kx , f^k q) < r\text{ for all } 0\leq k < m \}, \\
B_n^s(q,r) &= \{y\in W_q^s : d(f^{-k} y, f^{-k}q) < r\text{ for all } 0\leq k < n\}.
\end{align*}
Here we use the metric on $M$ rather than the induced leaf metric on $W_q^u$ and $W_q^s$;
this in particular yields the relationship $B_1^{s,u}(q,r) = B(q,r) \cap W^{s,u}_q$.
(One could also give the proof using the induced leaf metric, making minor changes to various constants.)
We will also need to consider the two-sided Bowen balls
\[
B_{n,m}(q,r) = \{z \in M : d(f^k z, f^k q) < r \text{ for all } -n < k < m \}.
\]
Observe that $B_{1,m}(q,r) = B_m(q,r)$, the usual one-sided Bowen ball in $M$; more generally, we have
\begin{equation}\label{eqn:Bnm-as-image}
B_{n,m}(q,r) = f^{n-1}(B_{m+n-1}(f^{-(n-1)}(q),r)).
\end{equation}
The next proposition uses the quantity $\eps_1$ defined in Lemma \ref{lem:bracket-inv}, and the function $\omega$ defined in Lemma \ref{lem:bracket-cts}.

\begin{proposition}\label{prop:Bow-balls-1}
Fix $r \in (0,\eps_1)$, and let $r_1 \in (0,r)$ be sufficiently small that $\omega(r_1) < r$. 
Suppose $x,y\in M$ are sufficiently close that $CQ d(x,y) + r_1 < \eps_1$.
Then for every $m,n>0$, we have
\begin{equation}\label{eqn:Bnm}
B_{n,m}([x,y],r_1) \subset [B_m^u(x,r), B_n^s(y,r)].
\end{equation}
\end{proposition}
\begin{proof}
Since $[x,y] \in W_x^s$, for every $k\geq 0$ we have
\begin{equation}\label{eqn:still-on-local}
d(f^k([x,y]), f^k x) \leq C d([x,y], x) \leq CQ d(x,y).
\end{equation}
Now given any $z\in B_{n,m}([x,y],r_1)$, for every $0\leq k < m$ we have
\begin{equation}\label{eqn:r1}
d(f^k z, f^k([x,y])) < r_1;
\end{equation}
see Figure \ref{fig:Bow-balls}(a). Together with \eqref{eqn:still-on-local}, this gives
\[
d(f^k z, f^k x) \leq CQ d(x,y) + r_1 < \eps_1.
\]
Thus we can apply Lemma \ref{lem:bracket-inv} to get $[f^k z, f^k x] = f^k[z,x]$ for all $0\leq k < m$. 
Meanwhile, observing that $f^k[x,y] \in W_{f^k x}^s$, we can use \eqref{eqn:r1} and Lemma \ref{lem:bracket-cts} to get
\[
d([f^k z, f^k x], f^k x) < \omega(r_1) < r.
\]
Putting these together gives $d(f^k[z,x], f^k x) < r$. This holds for all $0\leq k < m$, and since $[z,x] \in W_x^u$, this shows that $[z,x] \in B_m^u(x,r)$. A similar argument for $-n < k \leq 0$, with the roles of stable and unstable reversed, shows that $[q,z] \in B_n^s(y,r)$. Together these prove \eqref{eqn:Bnm}.
\end{proof}

\begin{figure}[tbp]
\begin{tikzpicture}[scale=.8, bend angle=8, inner sep=1pt, label distance=1pt,
					arr/.style={<-,draw=gray,shorten <=2pt, >=stealth},
					dot/.style={fill=blue,circle,minimum size=3pt}]
\coordinate (xo) at (0,0);

\path [name path=Wxu, draw]
	(xo) ++(-1,0) to [bend left] ++(1.5,0) to [bend right] ++(1.5,0);
\path [name path=Wxs, draw]
	(xo) ++(0,-1) to [bend left] ++(0,2) to [bend right] ++(0,2); 
\path [name intersections={of=Wxu and Wxs, by=x}];
\node[dot] at (x) [label=-135:$f^kx$] {};
\path [name path=Wzs, draw]
	(xo) ++(1,-1) to [bend left] ++(-.2,2) to [bend right] ++(-.2,2);
\path [name intersections={of=Wxu and Wzs, by=zx}];
\node[dot] at (zx) [label=-45:{$f^k[z,x]$}] {};
\path [name path=Ly] (-1,1.5)-- ++(4,0);
\path [name path=Lz] (-1,2)-- ++(4,0);
\path [name intersections={of=Wxs and Ly, by=xy}];
\path [name intersections={of=Wzs and Lz, by=z}];
\node[dot] at (xy) [label=left:{$f^k[x,y]$}] {};
\node[dot] at (z) [label=right:$f^kz$] {};

\draw[dotted] (xy)--(z);
\draw[dotted] (x)--(z);

\node[fill=green!50!white,rectangle,inner sep=2pt] (r1) at (-1,2.5) {$d< r_1$};
\node[fill=green!50!white,rectangle,inner sep=2pt] (eps1) at (3,2) {$d<\eps_1$};
\node[fill=green!50!white,rectangle,inner sep=2pt] (r) at (3,.5) {$d< r$};

\draw[arr] ($.5*(xy)+.5*(z)$) to [bend right=20](r1);
\draw[arr] ($.5*(x)+.5*(z)$) to [bend right=20](eps1);
\draw[arr] ($.5*(x)+.5*(zx)$) to [bend left=20](r);
\end{tikzpicture}
\hfill
\begin{tikzpicture}[scale=.6, bend angle=8, inner sep=1pt, label distance=1pt,
					arr/.style={<-,draw=gray,shorten <=2pt, >=stealth},
					dot/.style={fill=blue,circle,minimum size=3pt}]
\coordinate (xo) at (0,0);
\path [name path=Wxu, draw]
	(xo) ++(-1,0) to [bend left] ++(2,0) to [bend right] ++(2,0);
\path [name path=Wxs, draw]
	(xo) ++(0,-1) to [bend left] ++(0,2) to [bend right] ++(0,2); 
\path [name intersections={of=Wxu and Wxs, by=x}];
\node[dot] at (x) [label=-135:$f^kx$] {};
\path [name path=Wzs, draw]
	(xo) ++(.8,-1) to [bend left] ++(.3,2) to [bend right] ++(.3,2);
\path [name intersections={of=Wxu and Wzs, by=zx}];
\node[dot] at (zx) [label=-45:{$f^kx'$}] {};
\path [name path=Wyu, draw] 
	(-1,1.7) to [bend left] ++(2,0) to [bend right] ++(2,0);
\path [name path=Wzu, draw] 
	(-1,2.5) to [bend left] ++(2,0) to [bend right] ++(2,0);
\path [name intersections={of=Wxs and Wyu, by=xy}];
\path [name intersections={of=Wzs and Wzu, by=z}];
\node[dot] at (xy) [label=below left:{$f^k[x,y]$}] {};
\node[dot] at (z) [label=above right:{$f^k[x',y']$}] {};
\path [name intersections={of=Wxs and Wzu, by=xy'}];
\path [name intersections={of=Wzs and Wyu, by=x'y}];
\node[dot] at (xy') [label=above left:{$f^k[x,y']$}] {};
\node[dot] at (x'y) [label=below right:{$f^k[x',y]$}] {};

\draw[dotted] (xy)--(z);
\draw[dotted] (xy')--(zx);

\node[fill=green!50!white,rectangle,inner sep=2pt] (Comega) at (-4.5,2.5) {$d< C\omega(r)$};
\node[fill=green!50!white,rectangle,inner sep=2pt] (eps1) at (5,0.5) {$d<\eps_1$};
\node[fill=green!50!white,rectangle,inner sep=2pt] (r2) at (5,2.5) {$d< r_2$};
\node[fill=green!50!white,rectangle,inner sep=2pt] (r) at (-4,0.2) {$d< r<C\omega(r)$};

\draw[arr] ($.5*(xy)+.5*(xy')$) to [bend left=20](Comega);
\draw[arr] ($.5*(xy')+.5*(zx)$) to [bend right=20](eps1);
\draw[arr] ($.5*(xy)+.5*(z)$) to [bend right=20](r2);
\draw[arr] ($.5*(x)+.5*(zx)$) to [bend right=20](r);
\end{tikzpicture}
\\[1ex]
(a)\hspace{7cm}(b)\quad
\caption{Distance estimates in Propositions \ref{prop:Bow-balls-1} and \ref{prop:Bow-balls-2}}
\label{fig:Bow-balls}
\end{figure}

\begin{proposition}\label{prop:Bow-balls-2}
Fix $r \in (0, \eps_1/(CQ+1))$.
Let $r_2 = \omega(C\omega(r))$. 
Suppose $x,y\in M$ are sufficiently close that $CQ d(x,y)+ (CQ+1)r<\eps_1$.
Then for every $m,n\in \NN$, we have
\begin{equation}\label{eqn:Bmn-sub}
[B_m^u(x,r), B_n^s(y,r)] \subset B_{n,m}([x,y],r_2)
\end{equation}
\end{proposition}
\begin{proof}
Fix $x' \in B_m^u(x,r)$ and $y' \in B_n^s(y,r)$. The bracket $[x',y']$ exists because $d(x',y') < d(x,y) + 2r < \eps$. We will prove that $d(f^k[x',y'], f^k[x,y]) < r_2$ for all $0\leq k < m$; the result for $-n<k\leq 0$ will follow by reversing the roles of stable and unstable.

First observe that $d([x,y'],x) \leq Q d(x,y') < Q(d(x,y) + r)$, so for all $k\geq 0$ we have $d(f^k[x,y'], f^k x) < CQ(d(x,y) + r)$. Given $0\leq k < m$, we deduce that
\[
d(f^k[x,y'], f^k x') \leq d(f^k[x,y'], f^k x) + d(f^k x, f^k x')
< CQ(d(x,y) + r) + r < \eps_1; 
\]
see Figure \ref{fig:Bow-balls}(b). We conclude that
\begin{equation}\label{eqn:xyx'y'}
\begin{aligned}
[f^k x', f^k[x,y']] &= f^k [x', [x,y']] = f^k [x', y'],
&&\text{ by Lemma \ref{lem:bracket-inv},} \\
[f^k x, f^k[x,y]] &= f^k[x,y], &&\text{ since } f^k[x,y] \in W_{f^k x}^s.
\end{aligned}
\end{equation}
Observe that $d([x,y'],[x,y]) < \omega(r)$ by Lemma \ref{lem:bracket-cts}, and thus $d(f^k[x,y'], f^k [x,y]) < C\omega(r)$ since these points lie on the same local stable manifold. Since $d(f^k x', f^k x) < r < C\omega(r)$, we can apply Lemma \ref{lem:bracket-cts} again to the points in \eqref{eqn:xyx'y'} and obtain
\[
d(f^k[x,y], f^k[x',y']) < \omega(C\omega(r)) = r_2.
\]
This holds for $0\leq k < m$, and a similar argument with stable and unstable reversed proves it for $-n<k\leq 0$, so we have shown that $[x',y'] \in B_{n,m}([x,y],r_2)$.
\end{proof}

\section{Further Gibbs bounds}\label{sec:quasi-indep}

Our arguments in this section require the following property.

\begin{definition}\label{def:Bow-prop}
Given a compact metric space $M$ and a continuous map $f\colon M\to M$, a function $\ph\colon M\to \RR$ has the \emph{Bowen property at scale $r>0$} if there exists $L\in \RR$ such that for every $x\in M$, $n\in \NN$, and $y\in B_n(x,r)$, we have
\begin{equation}\label{eqn:Bow-prop}
|S_n\ph(x) - S_n\ph(y)| \leq L.
\end{equation}
\end{definition}

As mentioned in the introduction, this property is automatically satisfied whenever $\ph$ has a measure with the Gibbs property \eqref{eqn:gibbs}; note that the following result does not require $f$ to be Anosov.

\begin{lemma}\label{lem:Bow-prop}
Let $M$ be a compact metric space and $f\colon M\to M$ a continuous map. Suppose that $\ph\colon M\to \RR$ is such that there exists a Borel probability measure $\mu$ on $M$ satisfying the Gibbs property \eqref{eqn:gibbs} for some $P\in \RR$ and every $r\in (0, r_0]$. Then $\ph$ has the Bowen property at every scale $r\in (0,r_0)$.
\end{lemma}
\begin{proof}
Fix $\zeta \in (0, r_0 - r]$, so that $\mu$ has the Gibbs property at scale $\zeta$ and at scale $r+\zeta$.
Given $x\in M$ and $y\in B_n(x,r)$, we have 
\begin{align*}
\mu(B_n(y,r+\zeta)) &\leq K(r+\zeta) e^{S_n\ph(y) - nP}, \\
\mu(B_n(x,\zeta)) &\geq K(\zeta)^{-1} e^{S_n\ph(x) - nP}.
\end{align*}
Moreover, $B_n(x,\zeta) \subset B_n(y, r+\zeta)$, so this gives
\[
K(\zeta)^{-1} e^{S_n\ph(x) - nP} \leq K(r+\zeta) e^{S_n\ph(y) - nP},
\]
from which we deduce that $S_n\ph(x) - S_n\ph(y) \leq \log(K(\zeta)K(r+\zeta))$, and by symmetry we obtain the result.
\end{proof}

Returning to the Anosov setting, we can use the Bowen property to deduce Gibbs bounds involving leafwise Bowen balls. In what follows, we will use the notation
\begin{equation}\label{eqn:Sn-}
S_n^-\ph(y) = \sum_{k=0}^{n-1} \ph(f^{-k} y),
\end{equation}
and will write two-sided bounds of the form $K^{-1} C \leq A \leq KC$ as $A = K^{\pm1} C$.

\begin{proposition}\label{prop:us-Gibbs}
Let $f$ be $C^1$ Anosov and let $\mu$ be an $f$-invariant Borel probability measure satisfying the Gibbs property \eqref{eqn:gibbs}.  Let $r \in (0, \eps_1/(CQ+1))$ be sufficiently small that $\omega(C\omega(r)) < r_0$. 
Then there exists $K_0=K_0(r)>0$ such that given any $x,y\in M$ sufficiently close that $CQd(x,y) + (CQ+1)r<\eps_1$, and any $m,n\in\NN$, we have
\begin{equation}
\label{eqn:us-Gibbs}
\mu([B_m^u(x,r),B_n^s(y,r)])
= K_0^{\pm 1} e^{-mP + S_m\ph(x)} e^{-nP + S_n^-\ph(y)}.
\end{equation}
\end{proposition}

\begin{proof}
Let $r_2 = \omega(C\omega(r))$.
Using Proposition \ref{prop:Bow-balls-2} and invariance of $\mu$, we have
\begin{equation}\label{eqn:mu-Bnm}
\begin{aligned}
\mu([B_m^u(x,r),&B_n^s(y,r)])
\leq \mu(B_{n,m}([x,y],r_2)) \\
&= \mu(B_{m+n-1}(f^{-(n-1)}([x,y]),r_2)) &&\text{by \eqref{eqn:Bnm-as-image}} \\
&\leq K(r_2) e^{S_{m+n-1}\ph(f^{-(n-1)}([x,y])) - (m+n-1)P} &&\text{by \eqref{eqn:gibbs}} \\
&= K(r_2) e^P e^{-\ph([x,y])} e^{S_m \ph([x,y]) - mP} e^{S_n^- \ph([x,y]) - nP} \\
&\leq K(r_2) e^{P+\|\ph\|} L(r_2)^2 e^{S_m\ph(x) -mP} e^{S_n^-\ph(y)-nP} &&\text{by \eqref{eqn:Bow-prop}}.
\end{aligned}
\end{equation}
This proves one half of \eqref{eqn:us-Gibbs}. Replacing $r_2$ with $r_1\in (0,r)$ sufficiently small that $\omega(r_1) < r$, the other direction follows similarly using Proposition \ref{prop:Bow-balls-1}.
\end{proof}

\begin{proposition}\label{prop:leaf-Gibbs}
Let $f$ be $C^1$ Anosov and let $\mu$ be an $f$-invariant Borel probability measure satisfying the Gibbs property \eqref{eqn:gibbs}. Fix $0 < \bar{r}_1 \leq \bar{r}_2 < \eps_1/(2CQ+1)$ such that $\omega(C\omega(\bar{r}_2)) < r_0$. Then there exists $K_1 = K_1(\bar{r}_1)$ such that if $R\subset M$ is a rectangle and $q\in R$ a point such that
\begin{equation}\label{eqn:adapted-R}
\begin{aligned}
B_1^u(q,\bar{r}_1) \subset V_q^u &:= R \cap W_q^u \subset B_1^u(q,\bar{r}_2), \\
B_1^s(q,\bar{r}_1) \subset V_q^s &:= R \cap W_q^s \subset B_1^s(q,\bar{r}_2),
\end{aligned}
\end{equation}
then for every $r\in [\bar{r}_1, \bar{r}_2]$, $x\in V_q^u$ and $m\in \NN$ satisfying $B_m^u(x,r) \subset V_q^u$, we have
\begin{equation}\label{eqn:u-Gibbs}
\mu^u(B_m^u(x,r)) = K_1^{\pm 1} e^{-mP + S_m\ph(x)},
\end{equation}
and similarly for every $y\in V_q^s$ and $n\in \NN$ satisfying $B_n^s(y,r) \subset V_q^s$, we have
\begin{equation}\label{eqn:s-Gibbs}
\mu^s(B_n^s(y,r)) = K_1^{\pm 1} e^{-nP + S_n^- \ph(y)},
\end{equation}
where $\mu^{u,s}$ are the measures on $V_q^{u,s}$ defined in \eqref{eqn:musu}.
\end{proposition}
\begin{proof}
By the definition of $\mu^u$ and the assumption that $B_m^u(x,r) \subset V_q^u$, we have
\[
\mu^u(B_m^u(x,r)) = \mu([B_m^u(x,r), V_q^s]).
\]
By \eqref{eqn:adapted-R}, this gives
\[
\mu([B_m^u(x,\bar{r}_1),B_1^s(q,\bar{r}_1)]) \leq \mu^u(B_m^u(x,r))
\leq \mu([B_m^u(x,\bar{r}_2),B_1^s(q,\bar{r}_2)]),
\]
so \eqref{eqn:u-Gibbs} follows from Proposition \ref{prop:us-Gibbs} upon observing that $CQd(x,q) + (CQ+1)r \leq CQ\bar{r}_2 + (CQ+1)\bar{r}_2 < \eps_1$. The proof of \eqref{eqn:s-Gibbs} is similar.
\end{proof}

\section{Completion of the proof}\label{sec:finish}

Now we use Propositions \ref{prop:us-Gibbs} and \ref{prop:leaf-Gibbs} to prove \eqref{eqn:unif-bdd}, which will show Theorem \ref{thm:main}.
Let $r \in (0,\eps_1/(2CQ+1))$ be sufficiently small that $\omega(C\omega(r))<r_0$, and let $K_1 = K_1(r/2)$ be provided by Proposition \ref{prop:leaf-Gibbs} with $\bar{r}_1 = r/2$ and $\bar{r}_2 = r$. Then given $q\in M$, the rectangle $R = [B_1^u(q,r),B_1^s(q,r)]$ satisfies \eqref{eqn:adapted-R}, and thus we can use the bounds \eqref{eqn:u-Gibbs} and \eqref{eqn:s-Gibbs} on the leaf measures $\mu^{u,s}$.

Let $A^u \subset V_q^u$ and $A^s \subset V_q^s$ be any Borel sets. Let $K^{s,u} \subset A^{s,u}$ be arbitrary compact sets, and $G^{s,u} \supset A^{s,u}$ be arbitrary (relatively) open sets in $V_q^{s,u}$. By compactness and uniform hyperbolicity, there exists $n\in \NN$ such that for every $x\in K^u$, we have $B_n^u(x,r) \subset G^u$, and similarly for $K^s, B_n^s, G^s$. Let $E^u \subset K^u$ be a maximal $(n,r)$-separated set; that is, $E^u$ has the property that for every $x\in E^u$, we have $B_n^u(x,r) \cap E^u = \{x\}$, and it is not properly contained in any larger set with this property. It follows that:
\begin{itemize}
\item the sets $\{ B_n^u(x,r/2) : x\in E^u \}$ are disjoint and contained in $G^u$ (if they are not disjoint then $E^u$ is not $(n,r)$-separated, since $B_n^u(x,r/2) \cap B_n^u(y,r/2) \neq \emptyset$ implies $y\in B_n^u(x,r)$);
\item the sets $\{ B_n^u(x,r) : x\in E^u \}$ cover $K^u$ (otherwise maximality would be violated).
\end{itemize}
Choosing $E^s \subset K^s$ similarly, we obtain
\[
[G^u, G^s] \supset \bigsqcup_{(x,y) \in E^u\times E^s} [B_n^u(x,r/2), B_n^s(y,r/2)],
\]
and thus by Proposition \ref{prop:us-Gibbs}, 
\begin{equation}\label{eqn:mu-GuGs}
\mu([G^u,G^s]) \geq \sum_{(x,y) \in E^u\times E^s} K_0(r/2)^{-1} e^{-nP + S_n\ph(x)} e^{-nP + S_n^- \ph(y)}.
\end{equation}
Similarly,
\[
[K^u,K^s] \subset \bigcup_{(x,y) \in E^u\times E^s} [B_n^u(x,r), B_n^s(y,r)]
\]
and thus Proposition \ref{prop:us-Gibbs} gives
\begin{equation}\label{eqn:mu-KuKs}
\mu([K^u,K^s]) \leq \sum_{(x,y) \in E^u\times E^s} K_0(r) e^{-nP + S_n\ph(x)} e^{-nP + S_n^- \ph(y)}.
\end{equation}
On the unstable leaves, we have
\[
\bigsqcup_{x\in E^u} B_n^u(x,r/2) \subset G^u
\quad\text{and}\quad
K^u \subset \bigcup_{x\in E^u} B_n^u(x,r),
\]
which together with Proposition \ref{prop:leaf-Gibbs} gives
\begin{align}\label{eqn:Gu}
\mu^u(G^u) &\geq \sum_{x\in E^u} K_1^{-1} e^{-nP + S_n\ph(x)}, \\
\label{eqn:Ku}
\mu^u(K^u) &\leq \sum_{x\in E^u} K_1 e^{-nP + S_n\ph(x)}.
\end{align}
A similar argument on the stable leaves gives
\begin{align}\label{eqn:Gs}
\mu^s(G^s) &\geq \sum_{x\in E^s} K_1^{-1} e^{-nP + S_n\ph(x)}, \\
\label{eqn:Ks}
\mu^s(K^s) &\leq \sum_{x\in E^s} K_1 e^{-nP + S_n\ph(x)}.
\end{align}
Combining \eqref{eqn:mu-KuKs}, \eqref{eqn:Gu}, and \eqref{eqn:Gs} gives
\[
\mu([K^u, K^s]) \leq K_0(r/2) K_1^2 \mu^u(G^u) \mu^s(G^s).
\]
Taking a supremum over all compact $K^{u,s} \subset A^{u,s}$, and an infimum over all $G^{u,s} \supset A^{u,s}$, we see from regularity of the measures $\mu,\mu^u,\mu^s$ that
\[
\mu([A^u,A^s]) \leq K_0(r/2) K_1^2 \mu^u(A^u) \mu^s(A^s).
\]
A similar argument using \eqref{eqn:mu-GuGs}, \eqref{eqn:Ku}, and \eqref{eqn:Ks} provides the corresponding lower bound, so we have proved that taking $\bar{K} = K_0(r/2) K_1^2$, \eqref{eqn:unif-bdd} holds whenever $R$ satisfies \eqref{eqn:adapted-R} and $E\subset R$ has the form $E = [A^u, A^s]$. The extension to arbitrary Borel sets $E\subset R$ follows a standard argument: let $\AAA$ denote the set of all Borel sets $E\subset R$ such that \eqref{eqn:unif-bdd} holds; let $\mathcal{A}_0$ be the algebra of all finite unions of disjoint measurable sets of the form $[A^u,A^s]$, and observe that $\AAA_0 \subset \AAA$; finally, observe that $\mathcal{A}$ is a monotone class, so by the monotone class theorem it must contain the Borel $\sigma$-algebra.

This proves the result for all rectangles $R$ of the form $R = [B_1^u(q,r), B_1^s(q,r)]$, and since any rectangle of sufficiently small diameter is contained in a rectangle of this form, we conclude that $\mu$ has local product structure in the sense of Definition \ref{def:lps}, which completes the proof of Theorem \ref{thm:main}.

\bibliographystyle{amsalpha}
\bibliography{references}

\end{document}